\documentclass[reqno,twoside]{amsart}

\usepackage{graphicx}
\usepackage{multirow}
\usepackage{amsmath,amssymb,amsfonts}
\usepackage{amsthm}
\usepackage{mathrsfs}
\usepackage[title]{appendix}
\usepackage[table]{xcolor}
\usepackage{geometry}
\usepackage{textcomp}
\usepackage{manyfoot}
\usepackage{booktabs}
\usepackage{algorithm}
\usepackage{algorithmicx}
\usepackage{algpseudocode}
\usepackage{listings}
\usepackage{stmaryrd}
\usepackage{microtype}
\usepackage{enumitem}
\usepackage{hyperref}

\newtheorem{theorem}{Theorem}
\newtheorem{proposition}[theorem]{Proposition} 
\newtheorem{remark}{Remark}
\newtheorem{definition}{Definition}
\newtheorem{lemma}{Lemma}

\newtheoremstyle{custom}
  {3pt}   
  {3pt}   
  {}      
  {0pt}   
  {\bfseries} 
  {}      
  {.5em}  
  {\thmname{#1}~\thmnumber{#2}} 

\theoremstyle{custom}

\newtheorem{assumptionInternal}{\hspace{-0.4em}}  

\newenvironment{assumption}[1][]{%
  \begin{assumptionInternal}[#1]%
}{%
  \end{assumptionInternal}%
}

\newcommand{\inter}[1]{\llbracket #1 \rrbracket}
\newcommand{\MC}[1]{\mathcal{#1}}

\DisableLigatures{encoding = *, family = * }
\numberwithin{equation}{section}
\numberwithin{theorem}{section}
\numberwithin{remark}{section}
\numberwithin{definition}{section}
\numberwithin{lemma}{section}

\title[]{A Multi-Objective Optimization framework for Decentralized Learning with coordination constraints}

\author{Roberto Morales\textsuperscript{\,$\ast$}}  
\address{\textsuperscript{$\ast$}\, Chair of Computational Mathematics, DeustoTech, University of Deusto, Avenida de las Universidades 24, 48007 Bilbao, Basque Country, Spain} 
\email{roberto.morales@deusto.es}
\thanks{This project has received funding from the European Research Council (ERC) under the European Union’s Horizon 2030 research and innovation programme (grant agreement NO: 101096251-CoDeFeL). U. Biccari has been supported by the Grants TED2021-131390B-I00/AEI/10.13039/501100011033 DasEl and PID2023-146872OB-I00-DyCMaMod of MINECO, Spanish Government} 

\author{Umberto Biccari\textsuperscript{\,$\ast$}}  
\email{umberto.biccari@deusto.es}

\keywords{Multi-Objective optimization, decentralized learning, Pareto optimal solutions, Classification problems}
\subjclass[2020]{58E17, 90C15, 90C23, 90C26, 93A14}

\begin{document}

\begin{abstract}
This article introduces a generalized framework for Decentralized Learning formulated as a Multi-Objective Optimization problem, in which both distributed agents and a central coordinator contribute independent, potentially conflicting objectives over a shared model parameter space. Unlike traditional approaches that merge local losses under a common goal, our formulation explicitly incorporates coordinator-side criteria, enabling more flexible and structured training dynamics. To navigate the resulting trade-offs, we explore scalarization strategies—particularly weighted sums—to construct tractable surrogate problems. These yield solutions that are provably Pareto optimal under standard convexity and smoothness assumptions, while embedding global preferences directly into local updates. We propose a decentralized optimization algorithm with convergence guarantees, and demonstrate its empirical performance through simulations, highlighting the impact of the coordinator’s influence on local agent behavior. The proposed approach offers a principled and customizable strategy for balancing personalization, fairness, and coordination in decentralized learning systems.
\end{abstract} 

\maketitle

\section{Introduction}

The increasing complexity of large-scale systems — spanning networks of sensors, autonomous agents, or collaborative decision-making units — has brought renewed attention to Decentralized Learning (DL) and optimization. In such settings, multiple agents interact to solve a shared task while maintaining autonomy over their own data, resources, or objectives. This distributed architecture arises naturally in scenarios where central coordination is costly, impractical, or undesired, motivating algorithmic strategies that are scalable, robust, and aligned with local constraints. \cite{mcmahan2017communication,kairouz2021advances,li2020federated}. Beyond technical innovation, this transition also responds to rising societal demands for AI models that are both effective and ethically responsible \cite{floridi2018ai4people,raji2022actionable}.  

A central challenge in decentralized systems lies in reconciling the heterogeneous objectives of individual agents. Each participant may optimize task-specific goals informed for instance by local structural constraints, which need not align with a shared global interest. In many applications, this tension is further compounded by the presence of coordinating principles — such as fairness, structural regularization, or global consistency — that must be enforced without undermining agent autonomy \cite{li2020federated,li2019fair,reddi2020adaptive,yang2019federated}. While existing decentralized approaches often rely on aggregating local losses or gradients under the assumption of a common objective, they typically overlook the hierarchical and heterogeneous structure of real-world systems. In practice, a central authority — such as a regulatory body, infrastructure manager, or institutional coordinator — may impose distinct global objectives that cannot be adequately captured through simple averaging. This gap motivates a more general formulation, in which both agents and a central coordinator contribute independent (and potentially conflicting) objectives to a shared optimization problem. This setting has led to a growing body of work that extends decentralized paradigms along dimensions such as scalability, fairness, and collaborative robustness. These developments are comprehensively reviewed in recent surveys \cite{alsharif2024contemporary, liu2024recent, kairouz2021advances}, which also underscore open challenges in aligning agent-level and system-wide goals. A key insight emerging from this literature is that decentralized training often involves competing or incompatible objectives, making Multi-Objective Optimization (MOO) a natural and principled framework to formalize and analyze such trade-offs \cite{miettinen1999nonlinear,boyd2004convex}.

Rather than collapsing all objectives into a single aggregate loss, MOO allows for the explicit representation of conflicts and synergies among stakeholders, enabling more interpretable and controllable learning dynamics. In these settings, a single solution that simultaneously minimizes all objectives is generally unattainable. Instead, one seeks for so-called Pareto optimal solutions — those for which no objective can be improved without deteriorating at least one other. This notion captures the essential structure of DL problems where agent-level interests and global coordination goals may be at odds.
Far from being a limitation, the existence of a Pareto front reflects the richness of the problem structure, allowing decision-makers or coordinating entities to navigate among valid trade-offs depending on context, priorities, or policy constraints.

However, this richness also introduces computational and conceptual challenges, as the goal becomes not to find a single optimum, but to explore or select among a continuum of efficient solutions. To address this, a common approach is to apply scalarization techniques, which convert vector-valued objectives into scalar-valued surrogate problems. Among the most widely used are weighted sums and Chebyshev norms, which allow for the approximation of the Pareto front while leveraging classical optimization tools \cite{ma2020efficient, khorram2014numerical, jahn2009vector, miettinen1999nonlinear}. These methods enable controlled exploration of the trade-offs inherent in multi-agent systems and allow practitioners to embed coordinator-side priorities directly into the optimization process. The result is a flexible, computationally efficient strategy that aligns DL dynamics with structural or policy-driven objectives, while maintaining theoretical guarantees under standard convexity assumptions.

Beyond optimization, the interactions between decentralized agents often exhibit inherently strategic characteristics, where individual behaviors affect — and are affected by — the choices of others. This has motivated the use of game-theoretic formulations to model decentralized learning as a multi-agent system of incentives and feedback loops \cite{liu2024game, rosen1965existence, ratliff2016characterization}. In particular, potential games provide a compelling theoretical lens, revealing conditions under which equilibrium states exist and can be efficiently reached by distributed algorithms \cite{marden2009cooperative}. This perspective reinforces the view of decentralized collaborative learning as not only an optimization problem, but also a structured interaction among agents with both cooperative and competitive elements.

Building on these ideas, in this work we propose a generalized framework for DL with explicit multi-objective formulation: individual agents optimize local empirical objectives, while a central coordinator can impose additional structural criteria or regularization on the shared model parameter space. This naturally leads to scalarized surrogate problems that blend local and global objectives through a tunable parameter. Our approach extends classical decentralized averaging schemes by embedding coordinator-side preferences directly into local training through scalarization. We provide theoretical convergence guarantees and Pareto optimality under smoothness and convexity assumptions, echoing ideas from control-theoretic interpretations of deep learning \cite{agrachev2020control, geshkovski2022turnpike, trelat2025turnpike}.

Finally, through experiments on benchmark datasets, we demonstrate the empirical impact of the scalarization parameter on the trade-off between local accuracy and global coordination. Our results highlight that the proposed method offers a flexible and principled strategy to balance personalization, fairness, and robustness in decentralized learning environments.

\subsection{Outline}
This paper is organized as follows. In Section \ref{section:math} we present the mathematical formulation of our MOO problem and we introduce a scalarization-based formulation that reduces the MOO problem to a weighted single-objective one. We also propose a generalized decentralized optimization algorithm to solve the scalarized problem, together with a couple of convergence theorems. Section \ref{section:numerical:experiments} provides some numerical experiments on the MNIST dataset to evaluate the performance of the proposed method, and to assess the impact of the coordinator's objectives in the learning task. In Section \ref{section:Conclusions}, we summarize our conclusions and discuss future directions. Finally, Appendix \ref{section:proof:of:the:main:result} contains the detailed proof of the main theoretical results. 

\section{Mathematical formulation and main theoretical results}\label{section:math}

Let $M\in \mathbb{N}$ denote the number of agents contributing to the DL task. Each agent $i\in \inter{M}:=\{1,\ldots, M\}$ has a dataset 
\begin{align*}
    D^i:=\Big\{(x_l^i,y_l^i)\in \mathbb{R}^{d_1} \times \mathbb{R}^{d_2}\Big\}_{l=1}^{n_i},
\end{align*}
over which it trains a ML model $f_{\Theta}: \mathbb{R}^{d_1} \to \mathbb{R}^{d_2}$, where $\Theta\in \mathcal{U}$ denotes the (shared) learnable parameters belonging to a parameter's space $\mathcal{U}$. The training objective of agent $i$ is given by a differentiable empirical risk function $C_i :\mathcal{U} \to \mathbb{R}$. Moreover, for simplicity, we will assume that 
\begin{align*}
    n_1 = n_2 = \ldots = n_M,    
\end{align*}
i.e. each agent has the same amount of data.

Together with the agent-specific objectives $\{C_i\}_{i=1}^M$, we assume the presence of a central coordinator who contributes with a set of $N\in \mathbb{N}$ additional criteria, represented by continuous functions $S_j: \mathcal{U} \to \mathbb{R}$, for $j\in\inter{N}$. These objectives may reflect regularization terms, structural preferences, or constraints to guide global coordination, and are introduced to facilitate each agent's training task. 

In the context just described, we can formulate the learning task as a MOO problem in which all local and global objectives are considered simultaneously: 
\begin{align}\label{Multi-objective:problem}
    \min_{\Theta\in \mathcal{U}} \Big(C_1(\Theta),\ldots, C_M(\Theta) , S_1(\Theta), \ldots, S_N(\Theta) \Big).
\end{align}

Notice that, in general, there does not exist a single solution $\Theta^* \in \mathcal{U}$ that simultaneously minimizes all the objective functions in \eqref{Multi-objective:problem}. This reflects the inherent trade-offs and conflicts among competing goals, which make it impossible to perfectly satisfy every criterion at once. As a result, solutions must balance these objectives, motivating the need for principled MOO strategies. This motivates the introduction of the concept of Pareto optimality, which formalizes how solutions can achieve this balance when a single global minimum is unattainable.

\begin{definition}
Given the MMO problem \eqref{Multi-objective:problem}, we say that $\Theta^*\in \mathcal{U}$ is a
\begin{itemize}
    \item[1.] {\bf Pareto optimal solution} if it for all $\Theta\in \mathcal{U}$ with $\Theta\neq\Theta^*$ it holds that 
    \begin{displaymath}
        \begin{array}{ll}
            C_i(\Theta)\geq C_i(\Theta^*), & \text{ for all } i\in\inter{M}
            \\
            S_j(\Theta)\geq S_j(\Theta^*), & \text{ for all } j\in\inter{N}
        \end{array}, 
    \end{displaymath}
    with at least one strict inequality, that is either $C_{i_0}(\Theta) > C_{i_0}(\Theta^*)$ for some $i_0\in\inter{M}$ or $S_{j_0}(\Theta) > S_{j_0}(\Theta^*)$ for some $j_0\in\inter{N}$.
    
    \item[2.] {\bf Locally Pareto optimal solution} if there exists $\delta>0$ such that $\Theta^*$ is Pareto optimal in $\mathcal{U} \cap B(\Theta^*,\delta)$, where $B$ refers to the open ball using the euclidean norm.

    \item[3.] {\bf (Locally) weakly Pareto optimal solution} if there exists $\delta > 0$ such that for all $\Theta\in \mathcal{U}\cap B(\Theta^*, \delta)$ with $\Theta\neq\Theta^*$ it holds that 
    \begin{displaymath}
        \begin{array}{ll}
            C_i(\Theta)> C_i(\Theta^*), & \text{ for all } i\in\inter{M}
            \\
            S_j(\Theta)> S_j(\Theta^*), & \text{ for all } j\in\inter{N}
        \end{array}, 
    \end{displaymath}
    In other words, no other feasible point strictly improves all objectives simultaneously within a neighborhood of $\Theta^*$.
\end{itemize}      
\end{definition}

The existence of Pareto optimal solutions to \eqref{Multi-objective:problem} is a direct consequence of Weierstrass's theorem and can be guaranteed under mild assumptions, namely that $\mathcal{U}$ is a compact set and that the objective functions $\{C_i\}_{i=1}^M$ and $\{S_j\}_{J=1}^N$ are continuous. Moreover, we have the following result whose proof can be found in \cite{jahn2009vector,miettinen1999nonlinear}.
\begin{proposition}
Suppose that $\mathcal{U}$ is convex.
\begin{itemize}
    \item[1.] If both $C_i$ and $S_j$ are convex, for all $i\in \inter{M}$ and for all $j\in \inter{N}$, then every locally Pareto optimal solution is also Globally Pareto optimal.
    \item[2.] If both $C_i$ and $S_j$ are quasi-convex functions, with at least one strictly quasi-convex function \footnote{We recall that, given a convex set $\mathcal{U}$, a function $f:\mathcal{U}\to\mathbb{R}$ is said to be \textbf{quasi}-\textbf{convex} if for every $\alpha \in \mathbb{R}$ the sublevel set $\{x\in \mathcal{U}: f(x) \leq \alpha\}$ is convex. This is equivalent to the inequality 
    \begin{align*}
        f\big(\lambda x + (1-\lambda) y\big) \leq \max\{f(x), f(y)\}, \quad \text{ for all } x,y \in\mathcal{U} \text{ and } \lambda \in [0,1].
    \end{align*} 
    If the strict inequality holds, then the function is said to be \textbf{strictly quasi}-\textbf{convex}.}, then every locally Pareto optimal solution is also globally Pareto optimal. 
\end{itemize}
\end{proposition}

At the same time, it is well-known that, in general, Pareto optimal solutions are not unique (see \cite[Example 11.4]{jahn2009vector}). Because multiple trade-offs between competing objectives can exist, there is typically an entire set, known as the Pareto front, of solutions that cannot be improved in one objective without sacrificing performance in at least one other. This non-uniqueness is a defining feature of multi-objective problems and motivates the use of scalarization or preference-based methods to select among the possible efficient solutions.

\subsection{Scalarized optimization in a decentralized setting}\label{section:main:results}

Due to the non-uniqueness of solutions and the presence of Pareto fronts, solving a MOO problem is inherently challenging. In most cases, it is practically impossible to compute the entire set of optimal solutions $\Theta^*$ of \eqref{Multi-objective:problem}. To address these challenges, the optimal control literature has introduced the \textbf{scalarization-based approach}, which enables the computation of at least a subset of the optimal solutions to \eqref{Multi-objective:problem}.

The starting point for formalizing this scalarization-based approach is the following optimization problem, that merges agents and coordinator's objectives into a single surrogate loss:
\begin{align}\label{weighted:sum:problem}
    \min_{\Theta \in \mathcal{U}} \left\{\frac{1-\lambda}{M} \sum_{i=1}^M C_i(\Theta) + \frac{\lambda}{N}\sum_{j=1}^{N} S_j(\Theta) \right\}.
\end{align}

Here, the parameter $\lambda \in [0,1)$ serves as a weight that modulates the balance between local agent objectives and the coordinator's preferences. When $\lambda=0$, the influence of the coordinator vanishes and agents minimize their individual objectives independently. Conversely, as $\lambda \to 1^-$, the coordinator's influence dominates, effectively centralizing decision-making. This formulation enables a tunable degree of personalization, echoing recent developments in personalized learning \cite{kurilovas2019advanced,tan2022towards}, but framed here within a rigorous MOO context.

Problem \eqref{weighted:sum:problem} is known in the optimal control literature as the \textbf{weighting scalarization} of \eqref{Multi-objective:problem}. It allows to simplify the computation of Pareto-like solutions to \eqref{Multi-objective:problem}. In fact, we have the following result whose proof can be found in \cite{miettinen1999nonlinear}. 
\begin{proposition}
Suppose that $C_i$ and $S_j$ are continuous, for all $i\in \inter{M}$ and $j\in \inter{N}$. 
\begin{itemize}
    \item[1.] A solution of the weighting problem \eqref{weighted:sum:problem} is weakly-Pareto optimal. 
    \item[2.] Suppose that the problem \eqref{weighted:sum:problem} has a unique solution. Then, such a solution is Pareto optimal.
    \item[3.] Suppose that $C_i$ and $S_j$ are convex, for all $i\in \inter{M}$ and $j\in \inter{N}$. If $\Theta^*\in \mathcal{U}$ is a Pareto optimal solution of \eqref{Multi-objective:problem}, then there exists $\Theta^*$ is also a solution of \eqref{weighted:sum:problem}.
\end{itemize}
\end{proposition}

Our main concern in the present paper is to propose a computational strategy allowing to solve \eqref{weighted:sum:problem} efficiently. To this end, we shall first rewrite the problem in the equivalent form 
\begin{equation}\label{problem:weighted:Fi}
    \begin{array}{ll}
        \displaystyle\min_{\Theta\in \mathcal{U}} \frac{1-\lambda}{M}\sum_{i=1}^{M} F_i(\Theta),
        \\
        \displaystyle F_i(\Theta):= C_i(\Theta) + \alpha\sum_{j=1}^{N} S_j(\Theta), & \quad\text{ for all } i\in\inter{M},
        \\
        \displaystyle \alpha:=\frac{\lambda}{(1-\lambda)N}.
    \end{array} 
\end{equation}

This formulation \eqref{problem:weighted:Fi} bears resemblance to the FedAvg algorithm introduced in \cite{mcmahan2017communication}, yet it is notably more general. In FedAvg, the coordinator acts passively by aggregating agents updates without directly influencing their local objectives. In contrast, our framework grants the coordinator an active role by embedding its preferences into the local agents’ objectives through scalarization. This mechanism allows the coordinator to steer the training process while preserving the decentralized nature of local computations.

These mentioned similarities inspired us to propose Algorithm \ref{algo:DL} to solve \eqref{problem:weighted:Fi}. As we shall see in our simulation experiments in Section \ref{section:numerical:experiments}, this proposed approach indeed allows to compute optimal solutions to the MOO problem we are considering. 

\begin{algorithm}
\caption{Decentralized Learning algorithm for \eqref{problem:weighted:Fi}}\label{algo:DL}
    \begin{algorithmic}[1]
        \Require $\Theta^0\in\mathcal{U}$ - initial guess for the parameters

        $\quad\, T\in\mathbb{N}^*$ - total number of iterations
        
        $\quad\, \tau\in\mathbb{N}^*$ - total number of epochs for the agents' updates
        
        $\quad\, \eta>0$ - learning rate for the agents' updates
        
        $\quad\, \lambda\in [0,1)$ - agents-coordinator trade-off

        \Ensure $\Theta^T$ - optimized parameters

        \For{$t\in \inter{T}$}
            \For{$i\in \inter{M}$}
                \State{Set $\Theta^{t-1,0}_i\leftarrow \Theta^0$.}
                \For{$k\in \inter{\tau}$}
                    \State{Compute $g_i^{t-1,k-1}$: stochastic gradient of $C_i(\Theta_i^{t-1,k-1})$.} 
                    \State{Compute $h_i^{t-1,k-1}$: stochastic gradient of $\displaystyle\sum_{j=1}^N S_j(\Theta_i^{t-1,k-1})$.} 
                    \State{Update: \begin{align}\label{eq:update_algo} \displaystyle\Theta_i^{t-1,k}\leftarrow\Theta_i^{t-1,k-1} -\eta \left(g_i^{t-1,k-1} + \alpha h_i^{t-1,k-1} \right).\end{align}}
                \EndFor
            \EndFor
            \State{Set $\displaystyle\Theta^t = \sum_{i=1}^M \Theta^{t-1,\tau}_i$}            
        \EndFor               
    \end{algorithmic}
\end{algorithm}

Moreover, the convergence of Algorithm \ref{algo:DL} can be guaranteed under the following standard regularity assumptions for the objective functions $\{C_i\}_{i=1}^M$ and $\{S_j\}_{j=1}^N$.

\begin{assumption}\label{ass:A1}
The objectives $\{C_i\}_{i=1}^M$ and $\{S_j\}_{j=1}^N$ are convex and $L$-smooth, that is, there exist two constants $L_C,L_S>0$ such that 
\begin{displaymath}
    \text{for all } \Theta,\tilde{\Theta}\in \mathcal{U}, \quad 
    \begin{array}{ll}
        \|\nabla C_i(\Theta) - \nabla C_i(\tilde{\Theta})\| \leq L_C \|\Theta -\tilde{\Theta}\| & \text{ for all } i\in\inter{M}
        \\
        \|\nabla S_j(\Theta) - \nabla S_j(\tilde{\Theta})\|\leq L_S \|\Theta - \tilde{\Theta}\| & \text{ for all } j\in\inter{N}
    \end{array}. 
\end{displaymath}
\end{assumption}

\begin{assumption}\label{ass:A2}
For all $i\in\inter{M}$, $k\in\inter{\tau}$, and $t\in\inter{T}$, the agents' stochastic gradients fulfill the following properties:

\medskip
\noindent For all $\Theta_i^{t,k}\in \mathcal{U}$,     
\begin{equation}\label{ass:A3:cond:01}
    \mathbb{E} \left[g_i^{t,k}\big|\Theta_i^{t,k}\right] = \nabla C_i\left(\Theta_i^{t,k}\right) \quad\text{ and }\quad \mathbb{E} \left[h_i^{t,k}\big|\Theta_i^{t,k}\right] = \nabla S\left(\Theta_i^{t,k}\right);
\end{equation}

\medskip 
\noindent There exist $\sigma_C,\sigma_S>0$ such that 
\begin{equation}\label{ass:A3:cond:02}
    \mathbb{E}\left[\Big\|g_i^{t,k} - \nabla C_i(\Theta_i^{t,k}) \Big\|^2 \Big| \Theta_i^{t,k}\right]\leq \sigma_C^2 \quad\text{ and }\quad \mathbb{E}\left[\Big\|h_i^{t,k} - \nabla S(\Theta_i^{t,k})\Big\|^2 \Big| \Theta_i^{t,k}\right] \leq \sigma_S^2.
\end{equation}

\end{assumption}

\begin{assumption}\label{ass:A3}
There exists $\zeta>0$ such that
\begin{align}\label{eq:zeta}
    \max_{i\in\inter{M}} \sup_{\Theta\in \mathcal{U}} \left\| \nabla F_i(\Theta) - \nabla F(\Theta)  \right\| \leq \zeta.
\end{align}
\end{assumption}

In particular, we have the following convergence results, whose proof will be given in Appendix \ref{section:proof:of:the:main:result}.

\begin{theorem}\label{Theorem:Per:Round:Progress}
Suppose that Assumptions \ref{ass:A1}-\ref{ass:A2} are fulfilled with regularity constants $L_C,L_S,\sigma_C,\sigma_S$. Set 
\begin{align}\label{eq:L_def}
    L:= L_C+\alpha L_S
\end{align}
and
\begin{align}\label{eq:Sigma_def}
    \Sigma:= \sigma_C^2 +\alpha^2 N^2 \sigma_S^2,
\end{align}
with $\alpha$ as in \eqref{problem:weighted:Fi}, and take the learning rate 
\begin{align}\label{eq:LR}
    \eta \leq \dfrac{1}{4L},
\end{align}
For all $t\in\inter{T-1}\cup\{0\}$ and $k\in\inter{\tau}$, with $T,\tau\in\mathbb{N}^\ast$, define the averaged parameters
\begin{align}\label{eq:average} 
    \overline{\Theta}^{t,k}:=\dfrac{1}{M} \sum_{i=1}^{M} \Theta_i^{t,k}.
\end{align}

Let $\Theta^T\in\mathcal{U}$ be the output of Algorithm \ref{algo:DL}. Then, the following convergence estimate holds
\begin{align}\label{estimate:Lemma:1}
    \mathbb{E}\left[\dfrac{1}{\tau} \sum_{k=1}^{\tau} F(\overline{\Theta}^{t,k}) - F(\Theta^T) \bigg|\mathcal{F}^{t,0} \right] \leq & \dfrac{1}{2\eta \tau} \left(\Big\|\overline{\Theta}^{t,0} - \Theta^T\Big\|^2 - \mathbb{E} \left[\Big\|\overline{\Theta}^{t,\tau} - \Theta^T\Big\|^2 \Big| \mathcal{F}^{t,0} \right] \right) \notag 
    \\
    &+ \dfrac{L}{M\tau} \sum_{k=1}^{\tau} \sum_{i=1}^{M} \mathbb{E} \left[\Big\|\Theta_i ^{t,k} - \overline{\Theta}^{t,k}\Big\|^2 \Big| \mathcal{F}^{t,0}\right] +\dfrac{2\Sigma\eta}{M},
\end{align}
where $\mathcal{F}^{t,0}$ is the $\sigma$-field representing all the historical information up to the start of the $t^{\text{th}}$-round.
\end{theorem}

\begin{theorem}\label{main:theorem}
Suppose that Assumptions \ref{ass:A1}-\ref{ass:A2}-\ref{ass:A3} are fulfilled with regularity constants $L_C,L_S,\sigma_C,\sigma_S,\zeta$. Set $D:=\|\Theta^{0,0} - \Theta^*\|$ and take the learning rate $\eta$ as
\begin{align}
    \label{LR:main:thm}
    \eta =\frac{1}{4L\tau\sqrt{T}},
\end{align} 
with $T,\tau\in\mathbb{N}^\ast$ and $L$ as in \eqref{eq:L_def}. For all $t\in\inter{T-1}\cup\{0\}$ and $k\in\inter{\tau}$, let $\overline{\Theta}^{t,k}$ be the averaged parameter given by \eqref{eq:average}. Let $\Theta^T\in\mathcal{U}$ be the output of Algorithm \ref{algo:DL}. Then, the following convergence estimate holds
\begin{align}\label{estimate:main:result}    
    \mathbb{E} \left[\dfrac{1}{\tau T} \sum_{t=0}^{T-1} \sum_{k=1}^{\tau} F(\overline{\Theta}^{t,k}) - F(\Theta^T) \right] \leq \dfrac{2D^2L}{\sqrt{T}} + \dfrac{\Sigma}{2LM\tau\sqrt{T}}  + \dfrac{5\zeta^2}{8LT} +\dfrac{\Sigma}{4L\tau T},
\end{align} 
with $\zeta$ and $\Sigma$ and $\zeta$ given by \eqref{eq:zeta} and \eqref{eq:Sigma_def}, respectively.
\end{theorem}

The strategy and insights used in the proof of Theorems \ref{Theorem:Per:Round:Progress} and \ref{main:theorem} can be found in several articles in the ML community, see for instance \cite{karimireddy2019scaffold,khaled2020tighter,stich2018local}. Moreover, our results are in accordance with known convergence principles in stochastic optimization (see, e.g., \cite{bottou2018optimization}), and in particular highlight the importance of selecting a learning rate $\eta$ with reduces along the minimization process. 

In fact, Theorem \ref{Theorem:Per:Round:Progress} illustrates the interplay between the learning rate and the bounds on the variance of the stochastic directions. If there were no noise in the gradient computation, i.e. $\sigma_C=\sigma_S=0$ (which implies $\Sigma=0$ according to \eqref{eq:Sigma_def}), then at each iteration $t$ one could obtain linear convergence to the optimal value as the number of epochs $\tau$ increases. This is a standard result for the full gradient method with a sufficiently small positive step-size. On the other hand, when the gradient computation is noisy, i.e. when $\Sigma$ is large, one clearly loses this property. One can still use a fixed step-size and be sure that the expected objective values will converge linearly to a neighborhood of the optimal value, but, after some point, the noise in the gradient estimates prevent further progress. In this framework, to ensure full convergence of the algorithm, this growing noise $\Sigma$ must be compensated through the adoption of a reducing learning rate strategy.

\begin{remark}\label{rem:lambda}
As previously mentioned, in the original scalarized formulation \eqref{weighted:sum:problem}, the parameter $\lambda \in [0,1)$ measures the effects of the local agent objectives and the coordinator's suggestions on the global MOO task. In particular, when $\lambda \to 1^-$, the coordinator's influence dominates while the agents' contribution disappears, resulting in a training problem which is essentially not taking into account the data. In this situation, one may expect the training task to become more difficult, and our convergence Theorems \ref{Theorem:Per:Round:Progress} and \ref{main:theorem} partially support this observation.

Notice that, although not explicitly written in the statements of our theorems, the regularity constant $L$ and $\Sigma$ defined in \eqref{eq:L_def} and \eqref{eq:Sigma_def} depend explicitly on $\lambda$:
\begin{align*}
    L = L_C + \alpha L_S = L_C + \frac{\lambda L_s}{(1-\lambda)N} \quad\text{ and }\quad \Sigma = \sigma_C^2 + \alpha^2N^2\sigma_S^2 = \sigma_C^2 + \frac{\lambda^2\sigma_s^2}{(1-\lambda)^2}.
\end{align*}

In the limit $\lambda\to 1^-$, we can then see how these constants blow-up, which affects the convergence of Algorithm \ref{algo:DL} in two main aspects:
\begin{itemize}
    \item[1.] the learning rate \eqref{eq:LR} needed in Theorem \ref{Theorem:Per:Round:Progress} tends to zero, while the radius of convergence in \eqref{estimate:Lemma:1} deteriorates;
    \item[2.] to maintain stable the convergence estimate \eqref{estimate:main:result}, $T$ must be very large, roughly of the order of $\Sigma^2$. 
\end{itemize}

For the sake of completeness, we shall comment that these are simply heuristic considerations that may not be respected in practical implementations of our method. In fact, our simulation experiments in Section \ref{section:numerical:experiments} will show how, through a suitable selection of the coordinator's objectives, is it still possible to obtain good results even for large values of $\lambda$.
\end{remark}

\section{Experimental evaluation}\label{section:numerical:experiments}

To assess the empirical behavior of the proposed scalarized optimization framework, we perform a series of experiments addressing a classification task over the MNIST-digits dataset, consisting of grayscale images of handwritten digits from 0 to 9 (10 labels) with a size of $28\times 28$ pixels.  

We explore both statistically homogeneous (IID) and heterogeneous (non-IID) data distributions among agents. Our goal is to understand how the scalarization parameter $\lambda$, which balances local and global objectives, influences learning dynamics, convergence behavior, and generalization accuracy. These experiments offer insight into the role of coordination under varying levels of agent autonomy and data heterogeneity. 

Each agent trains a standard Convolutional Neural Network (CNN) whose architecture is described in Table \ref{tab:cnn-architecture}

\begin{table}[h!]
\centering
    \begin{tabular}{|c|c|c|c|l|}
    \hline
    \rowcolor{lightgray} \textbf{Layer \#} & \textbf{Layer type} & \textbf{Input shape} & \textbf{Output shape} & \textbf{Details} 
    \\
    \hline\hline 1 & Input & (1, 28, 28) & (1, 28, 28) & \begin{tabular}{l} Grayscale image \end{tabular}
    \\
    \hline 2 & Conv2D & (1, 28, 28) & (32, 28, 28) & \begin{tabular}{l} 32 filters \\ $3 \times 3$ kernel \\ padding=1 \\ ReLU activation \end{tabular}
    \\
    \hline 3 & MaxPooling2D & (32, 28, 28) & (32, 14, 14) & \begin{tabular}{l} $2 \times 2$ kernel \\ stride=2 \end{tabular}
    \\
    \hline 4 & Conv2D & (32, 14, 14) & (64, 14, 14) & \begin{tabular}{l} 64 filters \\ $3 \times 3$ kernel \\ padding=1 \\ ReLU \end{tabular}
    \\
    \hline 5 & MaxPooling2D & (64, 14, 14) & (64, 7, 7) & \begin{tabular}{l} $2 \times 2$ kernel \\ stride=2 \end{tabular}
    \\
    \hline 6 & Flatten & (64, 7, 7) & (3136,) & \begin{tabular}{l} Flatten to vector \end{tabular}
    \\
    \hline 7 & Dense & (3136,) & (128,) & \begin{tabular}{l} Fully connected \\ ReLU activation \end{tabular}
    \\
    \hline 8 & Output (Dense) & (128,) & (10,) & \begin{tabular}{l} Softmax activation \\ 10 classes \end{tabular} 
    \\
    \hline
    \end{tabular}
    \caption{Summary of the CNN architecture employed in our experiments.}\label{tab:cnn-architecture}
\end{table}

The experiments were conducted on a local machine running macOS 15.5. The system is equipped with an 8 core ARM64 CPU and 8 GB of RAM. The algorithm was implemented in Python 3.10.13 using PyTorch 2.5.1. GPU acceleration via CUDA was not available. However, as the machine is based on Apple Silicon, computations were executed on the CPU of the Metal Performance Shaders (MPS) backened, which is the default acceleration path for PyTorch on macOS systems without NVIDIA GPUs. The code to reproduce the examples can be found at \cite{Github}. 

\subsection{Experiment \#1: IID data distribution}

In this experiment, we explore the performance of the Algorithm \ref{algo:DL} under a homogeneous (IID) data distribution. In this setting, the dataset is partitioned among 5 agents, each one of them receiving 8000 images for training and 2000 for local validation, drawn randomly across all digit classes. We suppose the coordinator has the following single objective
\begin{align*}
    S_1(\Theta):=10^2 \|\Theta\|_2^2,\quad \text{ for all }\Theta\in \mathcal{U},
\end{align*}
where $\|\cdot\|_2$ denotes the discrete 2-norm.

We have set the local training epochs of each agent to $\tau=1$, to contain the risk of agents' drift. Moreover, throughout the training, local validation is performed at each round $t\in\inter{T}$ to monitor the progression of each agent's model using Accuracy and F1 Score, capturing both general correctness and class-wise balance of predictions. At the end of training, a final global evaluation is carried out using a test set of 10000 images, ensuring that this assessment is fully independent of any training or validation samples used during the process. The learning rate has been fixed to $\eta =0.001$.

In a first experiment, we have fixed the balancing parameter $\lambda$ in \eqref{problem:weighted:Fi} to the value $\lambda=0.87$. In light also of what we highlighted in Remark \ref{rem:lambda}, our concern here was to verify whether our algorithm was capable of achieving good classification performances for large values of $\lambda$, i.e., in the ``difficult'' regime in which the coordinator guidance dominates over the agent's training. The results of our simulations, illustrated in Tables \ref{tab:IID:accuracy}-\ref{tab:IID:F1:score} and Figure \ref{fig:exp1:lambda:0.87} show that this is indeed the case.

\begin{table}[!h]
    \centering
    \begin{tabular}{|c|c|c|c|c|c|c|} \hline 
        & \multicolumn{6}{|c|}{\cellcolor{lightgray}\textbf{Round} $t$}
        \\
        \hline \cellcolor{lightgray}\textbf{Validation accuracy} & $t=1$ & $t=10$ & $t=20$ & $t=30$ & $t=40$ & $t=50$
        \\ 
        \hline \hline Agent 1 & 0.3900 & 0.9367 & 0.9367 & 0.9429 & 0.9371 & 0.9396 
        \\ 
        \hline Agent 2 & 0.3804  & 0.9317  & 0.9346 & 0.9371 & 0.9317  & 0.9325 
        \\ 
        \hline Agent 3 & 0.3825  & 0.9342  & 0.9379& 0.9392 & 0.9417  & 0.9383 
        \\ 
        \hline Agent 4 & 0.3842  & 0.9267 & 0.9279 & 0.9329 & 0.9313  & 0.9317 
        \\ 
        \hline Agent 5 & 0.3842  & 0.9371  & 0.9383 & 0.9396 & 0.9387  & 0.9363 
        \\ 
        \hline
    \end{tabular}
    \caption{Validation accuracy of each agent after 1, 10, 20, 30, 40 and 50 rounds, with $\lambda=0.87$ and IID data distribution.}\label{tab:IID:accuracy}
\end{table}

\begin{table}[!h]
    \centering
    \begin{tabular}{|c|c|c|c|c|c|c|} \hline
        & \multicolumn{6}{|c|}{\cellcolor{lightgray}\textbf{Round} $t$}
        \\
        \hline \cellcolor{lightgray}\textbf{Validation F1 score} & $t=1$ & $t=10$ & $t=20$ & $t=30$ & $t=40$ & $t=50$
        \\ 
        \hline \hline Agent 1 & 0.3966 & 0.9359 & 0.9362 & 0.9422 & 0.9364 & 0.9390 
        \\ 
        \hline
        Agent 2 & 0.3912 & 0.9307 & 0.9337 & 0.9361 & 0.9307 & 0.9348 
        \\ 
        \hline
        Agent 3 & 0.3867 & 0.9331 & 0.9372 & 0.9384 & 0.9410 & 0.9377 
        \\ 
        \hline
        Agent 4 & 0.3929 & 0.9264 & 0.9280 & 0.9325 & 0.9383 & 0.9316 
        \\ 
        \hline
        Agent 5 & 0.3887 & 0.9366 & 0.9381 & 0.9391 & 0.9383 & 0.9358 
        \\ 
        \hline
    \end{tabular}
    \caption{Validation F1 score of each agent after 1, 10, 20, 30, 40 and 50 rounds, with $\lambda=0.87$ and IID data distribution.}\label{tab:IID:F1:score}
\end{table}

\begin{figure}[!h]
    \centering
    \includegraphics[width=1\linewidth]{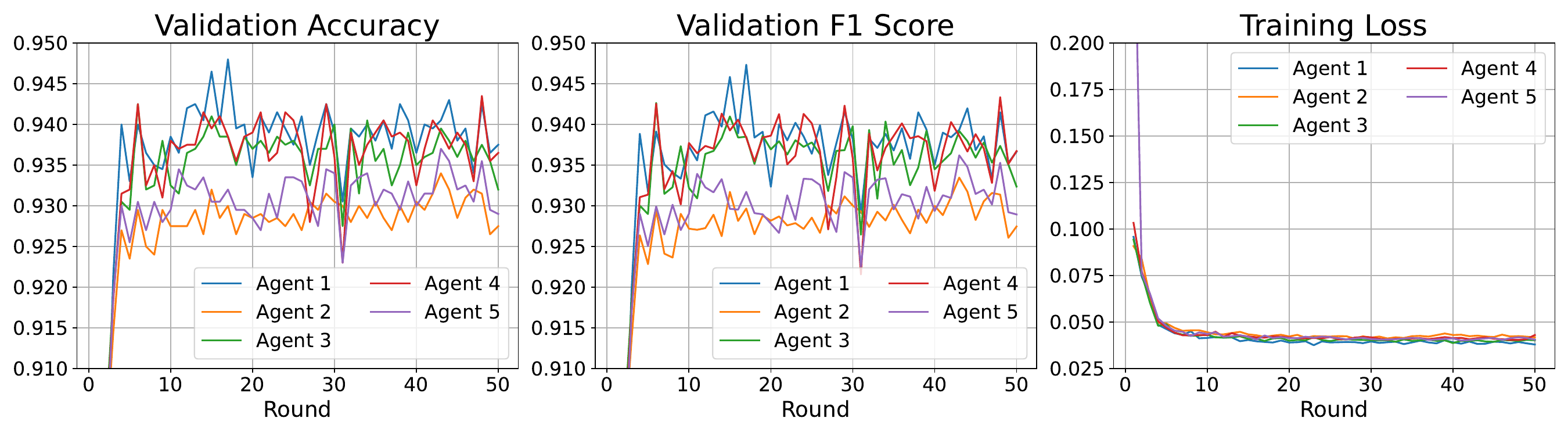}
    \caption{Evolution over 50 rounds of the agents' training performances with $\lambda=0.87$ and IID data distribution.}
    \label{fig:exp1:lambda:0.87}
\end{figure}

In fact, the validation accuracy and F1 score achieved by the model after 50 rounds are both above the threshold of $0.93$. While this is slightly below to the values one would expect when classifying MNIST using a CNN, we believe these results are still satisfactory given the limited influence the agents had in this training process. 

Moreover, these classification performances can be improved by reducing the value of $\lambda$, so that to give more importance to the agents' contribution. This phenomenon is illustrated in Table \ref{tab:IID:lambda}, collecting the performances on the test dataset of the model trained for 50 rounds, for different values of $\lambda\in [0, 1)$.

\begin{table}[!h]
    \centering
    \begin{tabular}{|c|c|c|c|c|c|c|} \hline 
        & \cellcolor{lightgray} $\lambda=0.00$ & \cellcolor{lightgray} $\lambda=0.25$ & \cellcolor{lightgray} $\lambda=0.50$ & \cellcolor{lightgray} $\lambda=0.65$ & \cellcolor{lightgray} $\lambda=0.75$ & \cellcolor{lightgray} $\lambda=0.87$
        \\ 
        \hline \hline \cellcolor{lightgray} \textbf{Test accuracy} & 0.9837 & 0.9846 & 0.9728 & 0.9657 & 0.9538 & 0.9409 
        \\ 
        \hline \cellcolor{lightgray} \textbf{Test F1 score} & 0.9836 & 0.9845 & 0.9727 & 0.9655 & 0.9535 & 0.9404 
        \\ 
        \hline 
    \end{tabular}
    \caption{Test accuracy and F1 score for different values of $\lambda$ after 50 rounds with IID data distribution.}
    \label{tab:IID:lambda}
\end{table}

This experiment highlights the delicate balance between enforcing global coordination and preserving agent-specific performance. Increasing the weight $\lambda$ on the coordinator's objective results in a measurable decline in accuracy and F1 score, reinforcing the necessity of careful tuning. The results demonstrate the trade-off inherent to multi-objective formulations, where global fairness or structure must be weighted against local adaptativity. These considerations are in agreement with Remark \ref{rem:lambda}, where we have pointed out how a large value of $\lambda$ disfavors the convergence of Algorithm \ref{algo:DL}. At this respect, we also point out that experiments we have conducted with $\lambda\geq 0.88$ have delivered negative results in which we are incapable of training the model. 

\subsection{Experiment \# 2: non-IID data distribution}

In this second experiment, we have considered a heterogeneous (non-IID) data distribution among the agents. In more detail, each of the five agents is assigned exactly two unique digit classes of MNIST:
\begin{itemize}
    \item Agent 1: digits 2 and 8;
    \item Agent 2: digits 4 and 9;
    \item Agent 3: digits 1 and 6;
    \item Agent 4: digits 3 and 7;
    \item Agent 5: digits 0 and 5.
\end{itemize}

From these two classes, each agent is assigned an equal and fixed number of samples (8000 training images and 2000 validation images) to ensure a balanced and fair local training environment. The class partitions are strictly non-overlapping, and thus the data distribution is highly non-IID across the agents. The coordinator's objective in this case is given by 
\begin{align*}
    S_1(\Theta):=10^7 \|\Theta\|_2^2, \quad\text{ for all } \Theta\in \mathcal{U}.
\end{align*}

As for the IID scenario before, we have set the local training epochs of each agent to $\tau=1$, to contain the risk of agents' drift. We mention that this is particular important for non-IID data distributions, since we want to avoid the agents to overfit on their incomplete dataset. Moreover, once again, we have conducted a first experiment with $\lambda=0.87$, to examine the performances in a ``difficult'' training regime. The results of this experiment are collected in Tables \ref{tab:NoIID:accuracy}-\ref{tab:NoIID:F1:score} and Figure \ref{fig:exp2:lambda:087}.

\begin{table}[!h]
    \centering
    \begin{tabular}{|c|c|c|c|c|c|c|} \hline 
        & \multicolumn{6}{|c|}{\cellcolor{lightgray}\textbf{Round} $t$}
        \\
        \hline \cellcolor{lightgray}\textbf{Validation accuracy} & $t=1$ & $t=10$ & $t=20$ & $t=30$ & $t=40$ & $t=50$
        \\ 
        \hline \hline Agent 1 & 0.4790 & 0.6450 & 0.6805 & 0.6330 & 0.6950 & 0.7670 
        \\ 
        \hline Agent 2 & 0.0000 & 0.2735 & 0.5040 & 0.7295 & 0.7875 & 0.7575 
        \\ 
        \hline Agent 3 & 0.0000 & 0.7390 & 0.8295 & 0.9195 & 0.9120 & 0.9015  
        \\ 
        \hline Agent 4 & 0.0000 & 0.2720 & 0.3875 & 0.4705 & 0.3915 & 0.4390 
        \\ 
        \hline Agent 5 & 0.0000 & 0.5610 & 0.6570 & 0.7215 & 0.7105 & 0.7160 
        \\ 
        \hline 
    \end{tabular}
    \caption{Validation accuracy of each agent after 1, 10, 20, 30, 40 and 50 rounds, with $\lambda=0.87$ and non-IID data distribution.}\label{tab:NoIID:accuracy}
\end{table}

\begin{table}[!h]
    \centering
    \begin{tabular}{|c|c|c|c|c|c|c|} \hline 
        & \multicolumn{6}{|c|}{\cellcolor{lightgray}\textbf{Round} $t$}
        \\
        \hline \cellcolor{lightgray}\textbf{Validation F1 score} & $t=1$ & $t=10$ & $t=20$ & $t=30$ & $t=40$ & $t=50$
        \\ 
        \hline \hline Agent 1 & 0.3239 & 0.1531 & 0.1601 & 0.1531 & 0.1633 & 0.1736 
        \\ 
        \hline Agent 2 & 0.0000 & 0.0783 & 0.1211 & 0.1683 & 0.1745 & 0.1692 
        \\ 
        \hline Agent 3 & 0.0000 & 0.1884 & 0.2015 & 0.2129 & 0.2726 & 0.2371 
        \\ 
        \hline Agent 4 & 0.0000 & 0.0854 & 0.1227 & 0.1228 & 0.1126 & 0.1209 
        \\ 
        \hline Agent 5 & 0000 & 0.1392 & 0.1550 & 0.1649 & 0.1630 & 0.1620 
        \\ 
        \hline        
    \end{tabular}
    \caption{Validation F1 score of each agent after 1, 10, 20, 30, 40 and 50 rounds, with $\lambda=0.87$ and non-IID data distribution.}\label{tab:NoIID:F1:score}
\end{table}

\begin{figure}[!h]
    \centering
    \includegraphics[width=1\linewidth]{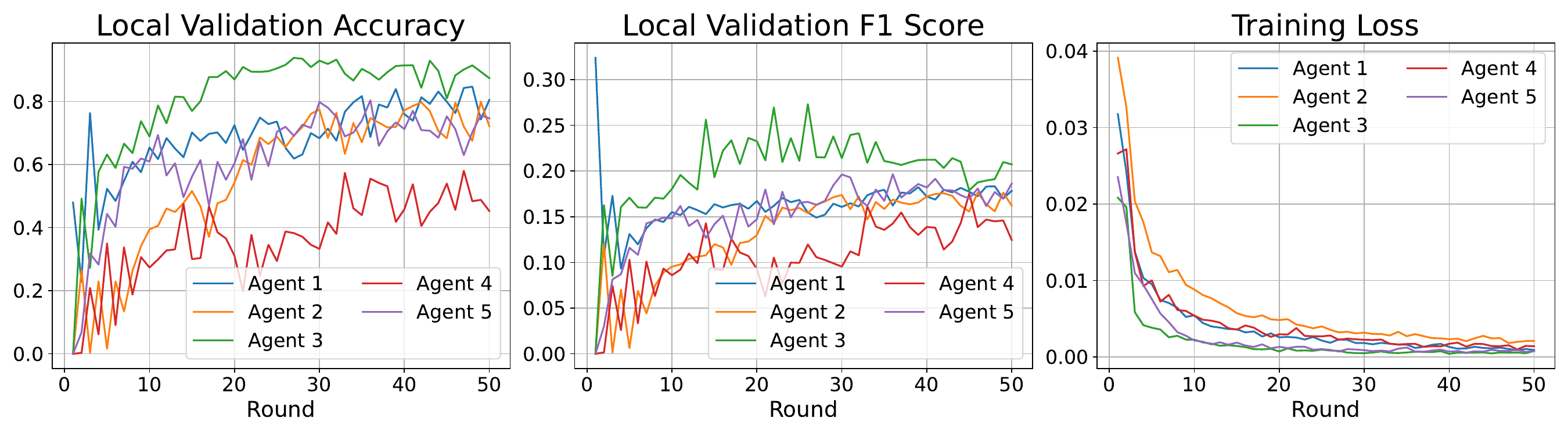}
    \caption{Evolution over 50 rounds of the agents' training performances with $\lambda=0.87$ and non-IID data distribution.}
    \label{fig:exp2:lambda:087}
\end{figure}

We can see how, this time, the individual performances of each agent are quite heterogeneous. At the same time, these performances tend to stabilize as the number of rounds increases and, in particular, each agent achieves a F1 score of around 0.2 at the end of the training. These values are in accordance with the data distribution, since the agents are all training over one fifth of the labels. Despite this behavior, the global model is capable of achieving 0.8056 accuracy and 0.7999 F1 score on the test dataset after 50 rounds. 

Finally, in Table \ref{tab:NoIID:lambda} we collect the test metrics after 50 rounds for different values of $\lambda\in [0,1)$. 

\begin{table}[!h]
    \centering
    \begin{tabular}{|c|c|c|c|c|c|c|} \hline 
        & \cellcolor{lightgray} $\lambda=0.00$ & \cellcolor{lightgray} $\lambda=0.25$ & \cellcolor{lightgray} $\lambda=0.50$ & \cellcolor{lightgray} $\lambda=0.65$ & \cellcolor{lightgray} $\lambda=0.75$ & \cellcolor{lightgray} $\lambda=0.87$
        \\ 
        \hline \hline \cellcolor{lightgray} \textbf{Test accuracy} & 0.6600 & 0.6943 & 0.7771 & 0.7212 & 0.7505 & 0.8056
        \\ 
        \hline \cellcolor{lightgray} \textbf{Test F1 score} & 0.6389 & 0.6820 & 0.7755 & 0.6919 & 0.7366 & 0.7999
        \\ 
        \hline 
    \end{tabular}
    \caption{Test accuracy and F1 score for different values of $\lambda$ after 50 rounds with non-IID data distribution.}
    \label{tab:NoIID:lambda}
\end{table}

Here, we can observe a completely different behavior than the IID case before (see Table \ref{tab:IID:lambda}). In particular, we notice how a larger $\lambda$ improves the training performances.

We believe this is consistent with our observations so far. In fact, when the data distribution is non-IID, agent drift becomes one of the main threats to overall training performance. If an agent overfits to its local dataset, which represents only a partial view of the underlying distribution, it may share misleading updates that ultimately harm the global model. Setting a large value of $\lambda$ helps mitigate this issue by reducing the weight of local training and shifting more influence to the coordinator. In doing so, the coordinator can guide the training process toward acceptable (albeit suboptimal) outcomes.

\section{Conclusions and open problems}\label{section:Conclusions}

In this article, we introduced a generalized multi-objective framework for DL, where several agents and a central coordinator jointly influence the training process via distinct objective functions. Our formulation allows for flexible integration of global structural preferences into the decentralized loop by means of scalarization techniques, which balance local performances with global alignment.

We proposed a scalarized optimization strategy that generalizes standard averaging-based algorithms and derived theoretical convergence guarantees under convexity and smoothness assumptions. Experiments on the MNIST dataset, both in IID and non-IID settings, showed how tuning the scalarization weight can shape learning dynamics, with clear implications for fairness, robustness, and personalization.

Our results demonstrate that DL systems can greatly benefit from incorporating multi-objective perspectives, especially when coordination needs to go beyond simple averaging. Still, our analysis has unaddressed several key questions which will be the objective of a future study.

\begin{itemize}
    \item[1.] \textbf{Optimality of the parameter $\lambda$.} In problem \eqref{weighted:sum:problem}, the parameter $\lambda \in [0,1)$ is introduced to balance the objectives of local agents with the coordinator's preferences. Our experiments reveal that the influence of this parameter is nuanced. Specifically, in the IID setting, increasing $\lambda$ leads to a mild degradation in training performance. However, in the non-IID case, this trend is reversed: a larger $\lambda$ proves beneficial in mitigating agent drift. These observations suggest that the optimal choice of $\lambda$ may depend heavily on the underlying data distribution, and that there could exist a value of $\lambda$ that yields the best overall performance. Determining whether such an optimal value exists, and characterizing it theoretically, remains a challenging question that lies beyond the scope of this study.
    \item[2.] \textbf{Non}-\textbf{smooth coordinator}'\textbf{s objectives.} In our work we have considered only the case in which the coordinator's objectives are smooth (see assumption \ref{ass:A1}). However, this is not the case in many real-life scenarios. In particular, the $L^1$-norm has proven highly effective in for inducing sparsity and enhancing intrepretability, yet its integration into decentralized and scalarized formulations poses significant challenges, such as lack of differentiability, difficulties in computing subgradients across agents, and the design of stable and convergent algorithms under stochastic updates. Addressing these issues requires extending the current theoretical tools, possibly by leveraging techniques from non-smooth analysis or proximal optimization. This remains an open and promising area for future research.
    \item[3.] \textbf{Chebyshev scalarization.} An alternative way to find weakly Pareto optimal solutions of \eqref{Multi-objective:problem} is to employ the so-called \textit{Chebyshev scalarization} which consists in optimizing a multi-objective function
    \begin{align*}
        H(\Theta)= \left(C_1(\Theta),\ldots, C_M (\Theta),S_1(\Theta),\ldots, S_N(\Theta) \right): \mathcal{U} \to \mathbb{R}^{M+N}
    \end{align*}
    by solving a min-max problem 
    \begin{align}\label{Chebyshev:formulation}
        \min_{\Theta\in \mathcal{U}} \max_{\rho \in \Delta} \rho^\top (H(\Theta) - s),
    \end{align}    
    where $s\in \mathbb{R}^{M+N}$ and $\Delta$ be the $(M+N)$-simplex. If \eqref{Chebyshev:formulation} admits a solution, then it is weakly Pareto optimal if all entries of $\rho$ are positive (see for instance \cite{miettinen1999nonlinear}). This formulation provides an alternative to the classical weighted-sum scalarization \eqref{weighted:sum:problem} by ensuring coverage of a broader portion of the Pareto front, especially in non-convex settings. However, solving the min-max problem \eqref{Chebyshev:formulation} in a decentralized and communication-efficient manner remains a significant challenge. In particular, the joint optimization over $\Theta$ and $\rho$ introduces nontrivial coupling between agents and coordinator, which complicates distributed implementation. While some centralized methods have been proposed in the literature for special cases, a fully decentralized algorithm that can solve \eqref{Chebyshev:formulation} under realistic constraints — such as limited communication, partial information, or stochastic gradients — remains an open problem. Addressing this challenge would enable more expressive and adaptive coordination strategies in multi-agent learning systems, particularly when strict prioritization or fairness constraints must be enforced across heterogeneous objectives.
\end{itemize}

\section*{Acknowledgements}
The authors would like to thank to Kang Liu (Institut de Math\'ematiques de Bourgogne), Ziqi Wang (Friedrich-Alexander-Universit\"at Erlangen-N\"urnberg ) and Enrique Zuazua (Friedrich-Alexander-Universit\"at Erlangen-N\"urnberg, University of Deusto and Universidad Aut\'onoma de Madrid) for their helpful feedback and stimulating conversations during the development of this research. 

\bibliographystyle{abbrv}
\bibliography{biblio}

{\appendix\section{Proof of the convergence results}\label{section:proof:of:the:main:result}

We present here the complete proofs for the two main theoretical result of the present paper, namely Theorems \ref{Theorem:Per:Round:Progress} and \ref{main:theorem}.

\begin{proof}[Proof of Theorem \ref{Theorem:Per:Round:Progress}]
Given the length and technicality of the proof, we shall split it into three steps to facilitate the reader's understanding. 

\medskip
\noindent\textbf{Step 1: effects of the parameter's update.} First of all recall that, for all $i\in\inter{M}$, $t\in\inter{T}$ and $k\in\inter{\tau}$, the parameters $\Theta_i^{t,k}$ are updated through the rule 
\begin{align*}
    \Theta_i^{t,k} = \Theta_i^{t,k-1} -\eta \MC{G}_i^{t,k-1},
\end{align*}
where we have defined 
\begin{align*}
    \MC{G}_i^{t,k-1}= g_i^{t,k-1} + \alpha h_i^{t,k-1}
\end{align*}
with $\alpha$ as in \eqref{problem:weighted:Fi}. By taking the average over $i\in\inter{M}$, we then have that 
\begin{align}\label{eq:update}
    \overline{\Theta}^{t,k}=\dfrac{1}{M} \sum_{i=1}^{M} \Theta_i^{t,k}=\dfrac{1}{M} \sum_{i=1}^{M} \left(\Theta_i^{t,k-1} -\eta \MC{G}_i^{t,k-1}\right) = \overline{\Theta}^{t,k-1} - \dfrac{\eta}{M} \sum_{i=1}^{M} \MC{G}_i^{t,k-1}.
\end{align}
Using \eqref{eq:update} and the parallelogram law, we can therefore rewrite
\begin{align}\label{Proof:Lemma:1:00}
    \left\|\overline{\Theta}^{t,k}\right. &- \left.\overline{\Theta}^{t,k-1}\right\|^2 + \left\|\overline{\Theta}^{t,k} - \Theta^T\right\|^2 \notag 
    \\
    &= \left\| \dfrac{\eta}{M} \sum_{i=1}^{M} \MC{G}_i^{t,k-1}\right\|^2 + \left\|\overline{\Theta}^{t,k-1} - \dfrac{\eta}{M} \sum_{i=1}^{M} \MC{G}_i^{t,k-1} - \Theta^T\right\|^2 \notag
    \\
    &= \left\| \dfrac{\sqrt{2}\eta}{M} \sum_{i=1}^{M} \MC{G}_i^{t,k-1}\right\|^2 + \left\|\overline{\Theta}^{t,k-1} - \Theta^T\right\|^2 - \frac{2\eta}{M}\sum_{i=1}^M \left\langle\MC{G}_i^{t,k-1},  \overline{\Theta}^{t,k-1} - \Theta^T\right\rangle.
\end{align}
Moreover, we can compute 
\begin{align*}
    \frac{2\eta}{M} \sum_{i=1}^M \bigg\langle\MC{G}_i^{t,k-1},  & \overline{\Theta}^{t,k-1} - \Theta^T\bigg\rangle - \left\| \dfrac{\sqrt{2}\eta}{M} \sum_{i=1}^{M} \MC{G}_i^{t,k-1}\right\|^2  
    \\
    &= \frac{2\eta}{M}\sum_{i=1}^M \left\langle\MC{G}_i^{t,k-1},  \overline{\Theta}^{t,k-1} - \Theta^T\right\rangle - \left\langle\dfrac{2\eta}{M} \sum_{i=1}^{M} \MC{G}_i^{t,k-1}, \dfrac{\eta}{M} \sum_{i=1}^{M} \MC{G}_i^{t,k-1}\right\rangle  
    \\
    &= \frac{2\eta}{M}\sum_{i=1}^M \left\langle\MC{G}_i^{t,k-1},  \overline{\Theta}^{t,k-1} - \Theta^T - \dfrac{\eta}{M} \sum_{i=1}^{M} \MC{G}_i^{t,k-1}\right\rangle 
\end{align*}
Using again \eqref{eq:update}, we hence have that 
\begin{align}\label{Proof:Lemma:1:01}
    \frac{2\eta}{M} & \sum_{i=1}^M \bigg\langle\MC{G}_i^{t,k-1},  \overline{\Theta}^{t,k-1} - \Theta^T\bigg\rangle - \left\| \dfrac{\sqrt{2}\eta}{M} \sum_{i=1}^{M} \MC{G}_i^{t,k-1}\right\|^2 = \frac{2\eta}{M}\sum_{i=1}^M \bigg\langle\MC{G}_i^{t,k-1},  \overline{\Theta}^{t,k} - \Theta^T\bigg\rangle.  
\end{align}
Combining \eqref{Proof:Lemma:1:01} with \eqref{Proof:Lemma:1:00}, we finally get the identity
\begin{align}\label{Proof:Lemma:1:020}
    \frac{1}{M}\sum_{i=1}^M \bigg\langle\MC{G}_i^{t,k-1}, \overline{\Theta}^{t,k} - \Theta^T\bigg\rangle = \frac{1}{2\eta}\Bigg(\Big\|\overline{\Theta}^{t,k-1} - \Theta^T\Big\|^2 - \Big\|\overline{\Theta}^{t,k}-\overline{\Theta}^{t,k-1}\Big\|^2 - \Big\|\overline{\Theta}^{t,k} - \Theta^T\Big\|^2\Bigg).
\end{align}

\medskip
\noindent\textbf{Step 2: estimates for the objective functions.} Notice that, by Assumptions \ref{ass:A1}-\ref{ass:A2}, the functions $\{F_i\}_{i=1}^M$ are all convex and $L$-smooth with $L$ given by \eqref{eq:L_def}. Then, for each $i\in\inter{M}$, we have 
\begin{align*}
    F_i(\overline{\Theta}^{t,k}) \leq &\; F_i(\Theta_i^{t,k-1}) + \left\langle \nabla F_i(\Theta_i^{t,k-1}) , \overline{\Theta}^{t,k} - \Theta_i^{t,k-1} \right\rangle 
    + \frac{L}{2} \left\|\overline{\Theta}^{t,k} - \Theta_i^{t,k-1}\right\|^2 
    \\
    \leq &\; F_i(\Theta^{t,k-1}) + \left\langle \nabla F_i(\Theta_i^{t,k-1}) , \Theta^T - \Theta_i^{t,k-1} \right\rangle 
    \\
    &+ \left\langle \nabla F_i(\Theta_i^{t,k-1}) , \overline{\Theta}^{t,k} - {\Theta}^T \right\rangle  
    + \frac{L}{2} \left\|\overline{\Theta}^{t,k} - \Theta_i^{t,k-1}\right\|^2 
    \\
    \leq &\; F_i(\Theta^T) + \left\langle \nabla F_i(\Theta_i^{t,k-1}), \overline{\Theta}^{t,k}-\Theta^T \right\rangle + L \left\|\overline{\Theta}^{t,k} - \overline{\Theta}^{t,k-1}\right\|^2 + L \left\|\Theta_i^{t,k-1}-\overline{\Theta}^{t,k-1}\right\|^2, 
\end{align*}
that is
\begin{align}\label{Proof:Lemma:1:02}
    F_i(\overline{\Theta}^{t,k}) - F_i(\Theta^T) \leq \left\langle \nabla F_i(\Theta_i^{t,k-1}), \overline{\Theta}^{t,k}-\Theta^T \right\rangle + L \left\|\overline{\Theta}^{t,k} - \overline{\Theta}^{t,k-1}\right\|^2 + L \left\|\Theta_i^{t,k-1}-\overline{\Theta}^{t,k-1}\right\|^2. 
\end{align}
Set 
\begin{align*}
    F(\cdot):=\frac{1}{M}\sum_{i=1}^M F_i(\cdot)
\end{align*}
By taking the average over $i\in\inter{M}$ in \eqref{Proof:Lemma:1:02}, we deduce that
\begin{align*}
    F (\overline{\Theta}^{t,k}) - F(\Theta^T) \leq& \dfrac{1}{M} \sum_{i=1}^{M} \left \langle \nabla F_i(\Theta_i^{t,k-1}), \overline{\Theta}^{t,k}-\Theta^T \right\rangle + \dfrac{L}{M} \sum_{i=1}^{M} \Big\|\Theta_i^{t,k-1} - \overline{\Theta}^{t,k-1}\Big\|^2 + L \Big\| \overline{\Theta}^{t,k} - \overline{\Theta}^{t,k-1}\Big\|^2  
    \\
    = & \dfrac{1}{M} \sum_{i=1}^{M} \left \langle \nabla F_i(\Theta_i^{t,k-1})-\MC{G}_i^{t,k-1}, \overline{\Theta}^{t,k} - \Theta^T \right\rangle 
    \\
    &+ \dfrac{1}{M} \sum_{i=1}^{M}\left \langle \MC{G}_i^{t,k-1}, \overline{\Theta}^{t,k} - \Theta^T \right\rangle  + \dfrac{L}{M} \sum_{i=1}^{M} \Big\|\Theta_i^{t,k-1} - \overline{\Theta}^{t,k-1}\Big\|^2 + L\Big\| \overline{\Theta}^{t,k} - \overline{\Theta}^{t,k-1}\Big\|^2. 
\end{align*}
Replacing \eqref{Proof:Lemma:1:020} in this last estimate, we then get
\begin{align}\label{Proof:Lemma:1:03}
    F (\overline{\Theta}^{t,k}) - F(\Theta^T) \leq & \dfrac{1}{M} \sum_{i=1}^{M} \left \langle \nabla F_i(\Theta_i^{t,k-1})-\MC{G}_i^{t,k-1}, \overline{\Theta}^{t,k}-\Theta^T \right\rangle \notag 
    \\
    &+ \frac{1}{2\eta}\Bigg(\Big\|\overline{\Theta}^{t,k-1} - \Theta^T\Big\|^2 - \Big\|\overline{\Theta}^{t,k}-\overline{\Theta}^{t,k-1}\Big\|^2 - \Big\|\overline{\Theta}^{t,k} - \Theta^T\Big\|^2\Bigg) \notag
     \\
    &+ \dfrac{L}{M} \sum_{i=1}^{M} \Big\|\Theta_i^{t,k-1} - \overline{\Theta}^{t,k-1}\Big\|^2 + L\Big\| \overline{\Theta}^{t,k} - \overline{\Theta}^{t,k-1}\Big\|^2.
\end{align}

\medskip
\noindent\textbf{Step 3: conclusion.} We shall now take the conditional expectation on the $\sigma$-field $\mathcal{F}^{t,k-1}$ of each term in \eqref{Proof:Lemma:1:03}. First of all, we can rewrite
\begin{align*}
    \mathbb{E} \Bigg[ \dfrac{1}{M} \sum_{i=1}^{M} & \left\langle \nabla F_i(\Theta_i^{t,k}) - \MC{G}_i^{t,k}, \overline{\Theta}^{t,k} - \Theta^T \right\rangle \bigg| \mathcal{F}^{t,k-1}  \Bigg]
    \\
    =&\, \mathbb{E} \left[ \dfrac{1}{M} \sum_{i=1}^{M} \left\langle \nabla F_i(\Theta_i^{t,k}) - \MC{G}_i^{t,k}, \overline{\Theta}^{t,k} - \overline{\Theta}^{t,k-1} \right\rangle \bigg| \mathcal{F}^{t,k-1}  \right]
    \\
    &+ \mathbb{E} \left[ \dfrac{1}{M} \sum_{i=1}^{M} \left\langle \nabla F_i(\Theta_i^{t,k}) - \MC{G}_i^{t,k}, \overline{\Theta}^{t,k-1} - \Theta^T \right\rangle \bigg| \mathcal{F}^{t,k-1}  \right].
\end{align*}
Moreover, thanks to Assumption \ref{ass:A3} (in particular, condition \eqref{ass:A3:cond:01}), we have that 
\begin{align*}
    \mathbb{E} \Bigg[ \dfrac{1}{M} \sum_{i=1}^{M} & \left\langle \nabla F_i(\Theta_i^{t,k}) - \MC{G}_i^{t,k}, \overline{\Theta}^{t,k-1} - \Theta^T \right\rangle \bigg| \mathcal{F}^{t,k-1}  \Bigg]
    \\
    =& \dfrac{1}{M} \sum_{i=1}^{M}\mathbb{E} \left[\left\langle \nabla F_i(\Theta_i^{t,k}) - \MC{G}_i^{t,k}, \overline{\Theta}^{t,k-1} - \Theta^T \right\rangle \bigg| \mathcal{F}^{t,k-1}  \right]
    \\
    =& \dfrac{1}{M} \sum_{i=1}^{M}\left\langle \mathbb{E} \left[\nabla F_i(\Theta_i^{t,k}) - \MC{G}_i^{t,k}\bigg| \mathcal{F}^{t,k-1}  \right], \overline{\Theta}^{t,k-1} - \Theta^T \right\rangle = 0.
\end{align*}
Hence, 
\begin{align}\label{Proof:Lemma:1:04}
    \mathbb{E} \Bigg[ \dfrac{1}{M} \sum_{i=1}^{M} & \left\langle \nabla F_i(\Theta_i^{t,k}) - \MC{G}_i^{t,k}, \overline{\Theta}^{t,k} - \Theta^T \right\rangle \bigg| \mathcal{F}^{t,k-1}  \Bigg] \notag
    \\
    =&\, \mathbb{E} \left[ \dfrac{1}{M} \sum_{i=1}^{M} \left\langle \nabla F_i(\Theta_i^{t,k}) - \MC{G}_i^{t,k}, \overline{\Theta}^{t,k} - \overline{\Theta}^{t,k-1} \right\rangle \bigg| \mathcal{F}^{t,k-1}  \right].
\end{align}
Applying Young's inequality, we then get
\begin{align}\label{Proof:Lemma:1:07}
    \mathbb{E} \Bigg[ &\dfrac{1}{M} \sum_{i=1}^{M} \left\langle \nabla F_i(\Theta_{i}^{t,k}) - \MC{G}_i^{t,k} , \overline{\Theta}^{t,k} - \Theta^T \right\rangle \bigg| \mathcal{F}^{t,k-1} \Bigg] \notag
    \\
    &= \mathbb{E} \Bigg[\dfrac{1}{M} \sum_{i=1}^{M} \left\langle \nabla F_i(\Theta_{i}^{t,k}) - \MC{G}_i^{t,k} , \overline{\Theta}^{t,k} - \overline{\Theta}^{t,k-1} \right\rangle \bigg| \mathcal{F}^{t,k-1} \Bigg] \notag
    \\
    &\leq \eta \mathbb{E} \left[\left\|\dfrac{1}{M} \sum_{i=1}^{M} \nabla F_i(\Theta_i^{t,k}) - \MC{G}_i^{t,k} \right\|^2 \bigg| \mathcal{F}^{t,k-1} \right] + \dfrac{1}{4\eta} \mathbb{E}\left[ \left\|\overline{\Theta}^{t,k} - \overline{\Theta}^{t,k-1}\right\|^2 \bigg| \mathcal{F}^{t,k-1}  \right]. 
\end{align}
Now, by means of \eqref{problem:weighted:Fi} and condition \eqref{ass:A3:cond:02}, we can easily estimate 
\begin{align*}
    \mathbb{E} \left[\left\|\dfrac{1}{M} \sum_{i=1}^{M}\nabla F_i(\Theta_i^{t,k}) - \MC{G}_i^{t,k} \right\|^2 \bigg| \mathcal{F}^{t,k-1} \right]\leq \frac{2(\sigma_C^2+\alpha^2N^2\sigma_S^2)}{M} = \frac{2\Sigma}{M}.
\end{align*}
Combining this last inequality with \eqref{Proof:Lemma:1:03}, \eqref{Proof:Lemma:1:04} and \eqref{Proof:Lemma:1:07}, we then see that 
\begin{align*}
    \mathbb{E} \Bigg[F(\overline{\Theta}^{t,k}) - F(\Theta^T) \bigg| \mathcal{F}^{t,k-1} \Bigg] \leq & \dfrac{2\Sigma\eta}{M} + \left(L - \dfrac{1}{4\eta} \right) \mathbb{E} \left[ \left\|\overline{\Theta}^{t,k} - \overline{\Theta}^{t,k-1} \right\|^2 \bigg| \mathcal{F}^{t,k-1} \right] 
    \\
    & +\dfrac{1}{2\eta} \mathbb{E} \Bigg[\left\| \overline{\Theta}^{t,k-1} - \Theta^T\right\|^2 - \left\|\overline{\Theta}^{t,k} - \Theta^T\right\|^2 \bigg| \mathcal{F}^{t,k-1} \Bigg] 
    \\
    &+\dfrac{L}{M} \sum_{i=1}^{M} \mathbb{E} \Bigg[ \left\| \Theta_i^{t,k-1} - \overline{\Theta}^{t,k-1} \right\|^2 \bigg| \mathcal{F}^{t,k-1} \Bigg].
\end{align*}
Therefore, since $\eta$ fulfills \eqref{eq:LR}, we can estimate 
\begin{align*}
    \mathbb{E} \Bigg[F(\overline{\Theta}^{t,k}) - F(\Theta^T) \bigg| \mathcal{F}^{t,k-1} \Bigg] \leq & \dfrac{2\Sigma\eta}{M} +\dfrac{1}{2\eta} \mathbb{E} \Bigg[\left\| \overline{\Theta}^{t,k-1} - \Theta^T\right\|^2 - \left\|\overline{\Theta}^{t,k} - \Theta^T\right\|^2 \bigg| \mathcal{F}^{t,k-1} \Bigg] 
    \\
    &+\dfrac{L}{M} \sum_{i=1}^{M} \mathbb{E} \Bigg[ \left\| \Theta_i^{t,k-1} - \overline{\Theta}^{t,k-1} \right\|^2 \bigg| \mathcal{F}^{t,k-1} \Bigg].
\end{align*}
Now, notice that 
\begin{align*}
    \left\| \overline{\Theta}^{t,k-1} - \Theta^T\right\|^2 \in \mathcal{F}^{t,k-1},
\end{align*}
which yields 
\begin{align*}
    \mathbb{E} \Bigg[\left\| \overline{\Theta}^{t,k-1} - \Theta^T\right\|^2 \bigg| \mathcal{F}^{t,k-1} \Bigg] = \left\| \overline{\Theta}^{t,k-1} - \Theta^T\right\|^2
\end{align*}
and
\begin{align*}
    \mathbb{E} \Bigg[F(\overline{\Theta}^{t,k}) - F(\Theta^T) \bigg| \mathcal{F}^{t,k-1} \Bigg] \leq & \dfrac{2\Sigma\eta}{M} +\dfrac{1}{2\eta} \left(\left\| \overline{\Theta}^{t,k-1} - \Theta^T\right\|^2 - \mathbb{E} \Bigg[\left\|\overline{\Theta}^{t,k} - \Theta^T\right\|^2 \bigg| \mathcal{F}^{t,k-1} \Bigg]\right) 
    \\
    &+\dfrac{L}{M} \sum_{i=1}^{M} \mathbb{E} \Bigg[ \left\| \Theta_i^{t,k-1} - \overline{\Theta}^{t,k-1} \right\|^2 \bigg| \mathcal{F}^{t,k-1} \Bigg].
\end{align*}
Moving to the left the first term on the right-hand side of the above inequality, we get
\begin{align}\label{Proof:Lemma:1:08}
    \mathbb{E} \Bigg[F(\overline{\Theta}^{t,k}) - F(\Theta^T) \bigg| \mathcal{F}^{t,k-1} \Bigg] +& \dfrac{1}{2\eta} \left( \mathbb{E} \left[ \left\|\overline{\Theta}^{t,k} - \Theta^T \right\|^2 \bigg| \mathcal{F}^{t,k-1} \right] - \left\|\overline{\Theta}^{t,k-1} - \Theta^T\right\|^2 \right) \notag
    \\
    &\leq \dfrac{2\Sigma\eta}{M} + \dfrac{L}{M} \sum_{i=1}^{M} \mathbb{E} \Bigg[ \left\|\Theta_i^{t,k-1} - \overline{\Theta}^{t,k-1}\right\|^2 \bigg| \mathcal{F}^{t,k} \Bigg].    
\end{align}
Multiplying by $\tau^{-1}$ in \eqref{Proof:Lemma:1:08} and summing from $k=1$ to $k=\tau$, we then obtain 
\begin{align*}
    \mathbb{E} \Bigg[\dfrac{1}{\tau} \sum_{k=1}^{\tau} \left( F(\overline{\Theta}^{t,k}) - F(\Theta^T) \right) \bigg| \mathcal{F}^{t,k-1} \Bigg] +& \dfrac{1}{2\eta \tau} \left( \mathbb{E}\left[ \left\|\overline{\Theta}^{t,\tau} - \Theta^T\right\|^2 \bigg| \mathcal{F}^{t,\tau-1} \right] - \left\|\overline{\Theta}^{t,0} - \Theta^T\right\|^2 \right) \notag
    \\
    &\leq \dfrac{2\Sigma\eta}{M} +\dfrac{L}{M\tau} \sum_{k=1}^{\tau} \sum_{i=1}^{M} \mathbb{E} \left[\left\|\Theta_i^{t,k-1} - \overline{\Theta}^{t,k-1}\right\|^2 \bigg| \mathcal{F}^{t,k} \right].
\end{align*}

From this last inequality, \eqref{estimate:Lemma:1} follows immediately using the Law of Iterated Expectations together with the clock property.
\end{proof}

The proof of our second main result Theorem \ref{main:theorem} requires the following technical lemma quantifying the so-called agent's drift, that is, the divergence of each agent's updates due to differences in their data distributions, which can hinder global model convergence and reduce overall performance.

\begin{lemma}\label{lemma:client_drift}
Suppose the assumptions \ref{ass:A1}-\ref{ass:A2}-\ref{ass:A3} are fulfilled and take the learning rate $\eta$ satisfying \eqref{eq:LR}. Then, for all $i\in\inter{M}$, $k\in\inter{\tau}$ and $t\in\inter{T-1}\cup\{0\}$, the following upper bound holds
\begin{align}\label{client:drift:inequality}
    \mathbb{E} \left[ \|\Theta_i^{t,k} - \overline{\Theta}^{t,k}\|^2 \Big| \mathcal{F}^{t,0} \right] \leq 10\tau^2 \eta^2 \zeta^2 + 4\tau \eta^2\Sigma,
\end{align}
with $\Sigma$ given by \eqref{eq:Sigma_def}.
\end{lemma}

\begin{proof}
Also in this case, to help the reader's understanding, we split the proof into two Steps. 

\medskip
\noindent\textbf{Step 1.} First of all, recall that we have defined 
\begin{align*}
    \MC{G}_i(\Theta):=g_i(\Theta) + \alpha \sum_{j=1}^{N} h_j(\Theta), \quad\text{ for all } i\in\inter{M}\text{ and } \Theta\in\mathcal{U}.
\end{align*}
Let now $i\in\inter{M-1}$. Then, by stochastic iterates \eqref{eq:update_algo}, and since 
\begin{align*}
    \left\|\Theta_i^{t,k} - \Theta_{i+1}^{t,k}\right\|^2\in \mathcal{F}^{(t,k)},     
\end{align*}
we have 
\begin{align}\label{Proof:Lemma:2:01}
    \mathbb{E} \left[ \left\|\Theta_i^{t,k+1} - \Theta_{i+1}^{t,k+1}\right\|^2 \big| \mathcal{F}^{t,k} \right] =& \mathbb{E}\left[\left\|\Theta_i^{t,k} - \Theta_{i+1}^{t,k} - \eta \left(\MC{G}_i(\Theta_i^{t,k}) - \MC{G}_{i+1}(\Theta_{i+1}^{t,k})\right)\right\|^2 \Big| \mathcal{F}^{(t,k)} \right] \notag
    \\
    =& \left\|\Theta_{i}^{t,k} - \Theta_{i+1}^{t,k}\right\|^2 + \eta^2 \mathbb{E} \left[ \left\|\MC{G}_i(\Theta_i^{t,k}) - \MC{G}_{i+1} (\Theta_{i+1}^{t,k})\right\|^2 \Big| \mathcal{F}^{t,k} \right] \notag 
    \\
    &-2\eta \mathbb{E}\bigg[ \left\langle  \MC{G}_i(\Theta_i^{t,k}) - \MC{G}_i(\Theta_{i+1}^{t,k}) , \Theta_i^{t,k} - \Theta_{i+1}^{t,k} \right\rangle \Big| \mathcal{F}^{t,k} \bigg].
\end{align}
Now, using the triangle inequality, we can split
\begin{align*}
    \eta^2\mathbb{E} \bigg[ & \left\|\MC{G}_i(\Theta_i^{t,k}) - \MC{G}_{i+1} (\Theta_{i+1}^{t,k})\right\|^2 \Big| \mathcal{F}^{t,k} \bigg] 
    \\
    =&\eta^2\mathbb{E} \left[ \left\|\MC{G}_i(\Theta_i^{t,k}) - \nabla F_i(\Theta_i^{t,k}) - \left(\MC{G}_{i+1} (\Theta_{i+1}^{t,k})-\nabla F_{i+1}(\Theta_{i+1}^{t,k})\right)\right\|^2 \Big| \mathcal{F}^{t,k} \right]
    \\
    &+\eta^2\mathbb{E} \left[ \left\|\nabla F_i(\Theta_i^{t,k}) - \nabla F_{i+1}(\Theta_{i+1}^{t,k})\right\|^2 \Big| \mathcal{F}^{t,k} \right]
    \\
    &+2\eta^2\mathbb{E} \bigg[ \left\langle \MC{G}_i(\Theta_i^{t,k}) - \nabla F_i(\Theta_i^{t,k}), \nabla F_i(\Theta_i^{t,k}) - \nabla F_{i+1}(\Theta_{i+1}^{t,k})\right\rangle \Big| \mathcal{F}^{t,k} \bigg]
    \\
    &-2\eta^2\mathbb{E} \bigg[\left\langle\MC{G}_{i+1} (\Theta_{i+1}^{t,k})-\nabla F_{i+1}(\Theta_{i+1}^{t,k}), \nabla F_i(\Theta_i^{t,k}) - \nabla F_{i+1}(\Theta_{i+1}^{t,k})\right\rangle|^2 \Big| \mathcal{F}^{t,k} \bigg].
\end{align*}
Moreover, thanks to \eqref{ass:A3:cond:01}, the last two terms in the last inequality vanish, so that we have 
\begin{align*}
    \eta^2\mathbb{E} \bigg[ & \left\|\MC{G}_i(\Theta_i^{t,k}) - \MC{G}_{i+1} (\Theta_{i+1}^{t,k})\right\|^2 \Big| \mathcal{F}^{t,k} \bigg] 
    \\
    =&\eta^2\mathbb{E} \left[ \left\|\MC{G}_i(\Theta_i^{t,k}) - \nabla F_i(\Theta_i^{t,k}) - \left(\MC{G}_{i+1} (\Theta_{i+1}^{t,k})-\nabla F_{i+1}(\Theta_{i+1}^{t,k})\right)\right\|^2 \Big| \mathcal{F}^{t,k} \right]
    \\
    &+\eta^2\mathbb{E} \left[ \left\|\nabla F_i(\Theta_i^{t,k}) - \nabla F_{i+1}(\Theta_{i+1}^{t,k})\right\|^2 \Big| \mathcal{F}^{t,k} \right]
\end{align*}
and \eqref{Proof:Lemma:2:01} becomes
\begin{align}\label{Proof:Lemma:2:02}
    \mathbb{E} \bigg[ \Big\|\Theta_i^{t,k+1} &- \Theta_{i+1}^{t,k+1}\Big\|^2 \big| \mathcal{F}^{t,k} \bigg] \notag
    \\
    =& \left\|\Theta_{i}^{t,k} - \Theta_{i+1}^{t,k}\right\|^2 + \eta^2\mathbb{E} \left[ \left\|\nabla F_i(\Theta_i^{t,k}) - \nabla F_{i+1}(\Theta_{i+1}^{t,k})\right\|^2 \Big| \mathcal{F}^{t,k} \right] \notag 
    \\
    &+\eta^2\mathbb{E} \left[ \left\|\MC{G}_i(\Theta_i^{t,k}) - \nabla F_i(\Theta_i^{t,k}) - \left(\MC{G}_{i+1} (\Theta_{i+1}^{t,k})-\nabla F_{i+1}(\Theta_{i+1}^{t,k})\right)\right\|^2 \Big| \mathcal{F}^{t,k} \right] \notag
    \\
    &-2\eta \mathbb{E}\bigg[ \left\langle  \MC{G}_i(\Theta_i^{t,k}) - \MC{G}_i(\Theta_{i+1}^{t,k}) , \Theta_i^{t,k} - \Theta_{i+1}^{t,k} \right\rangle \Big| \mathcal{F}^{t,k} \bigg].
\end{align}
In the same way, we have
\begin{align*}
    \mathbb{E}\bigg[\Big\langle  \MC{G}_i(\Theta_i^{t,k}) - \MC{G}_i & (\Theta_{i+1}^{t,k}) , \Theta_i^{t,k} - \Theta_{i+1}^{t,k} \Big\rangle \Big| \mathcal{F}^{t,k} \bigg] 
    \\
    =& \mathbb{E}\bigg[ \left\langle  \MC{G}_i(\Theta_i^{t,k}) - \nabla F_i(\Theta_i^{t,k}) , \Theta_i^{t,k} - \Theta_{i+1}^{t,k} \right\rangle \Big| \mathcal{F}^{t,k} \bigg]
    \\
    &- \mathbb{E}\bigg[ \left\langle  \MC{G}_i(\Theta_{i+1}^{t,k}) - \nabla F_{i+1}(\Theta_i^{t,k}) , \Theta_i^{t,k} - \Theta_{i+1}^{t,k} \right\rangle \Big| \mathcal{F}^{t,k} \bigg]
    \\
    &+ \mathbb{E}\bigg[ \left\langle  \nabla F_i(\Theta_i^{t,k}) - \nabla F_{i+1}(\Theta_{i+1}^{t,k}) , \Theta_i^{t,k} - \Theta_{i+1}^{t,k} \right\rangle \Big| \mathcal{F}^{t,k} \bigg]
    \\
    =& \mathbb{E}\bigg[ \left\langle  \nabla F_i(\Theta_i^{t,k}) - \nabla F_{i+1}(\Theta_{i+1}^{t,k}) , \Theta_i^{t,k} - \Theta_{i+1}^{t,k} \right\rangle \Big| \mathcal{F}^{t,k} \bigg]
    \\
    =& \left\langle  \nabla F_i(\Theta_i^{t,k}) - \nabla F_{i+1}(\Theta_{i+1}^{t,k}) , \Theta_i^{t,k} - \Theta_{i+1}^{t,k} \right\rangle,
\end{align*}
and we get from \eqref{Proof:Lemma:2:03}
\begin{align}\label{Proof:Lemma:2:03}
    \mathbb{E} \bigg[ \Big\|\Theta_i^{t,k+1} &- \Theta_{i+1}^{t,k+1}\Big\|^2 \big| \mathcal{F}^{t,k} \bigg] \notag
    \\
    =& \left\|\Theta_{i}^{t,k} - \Theta_{i+1}^{t,k}\right\|^2 + \eta^2\mathbb{E} \left[ \left\|\nabla F_i(\Theta_i^{t,k}) - \nabla F_{i+1}(\Theta_{i+1}^{t,k})\right\|^2 \Big| \mathcal{F}^{t,k} \right] \notag 
    \\
    &+\eta^2\mathbb{E} \left[ \left\|\MC{G}_i(\Theta_i^{t,k}) - \nabla F_i(\Theta_i^{t,k}) - \left(\MC{G}_{i+1} (\Theta_{i+1}^{t,k})-\nabla F_{i+1}(\Theta_{i+1}^{t,k})\right)\right\|^2 \Big| \mathcal{F}^{t,k} \right] \notag
    \\
    &-2\eta \left\langle  \nabla F_i(\Theta_i^{t,k}) - \nabla F_{i+1}(\Theta_{i+1}^{t,k}) , \Theta_i^{t,k} - \Theta_{i+1}^{t,k} \right\rangle.
\end{align}

\medskip
\noindent\textbf{Step 2.} Using \eqref{ass:A3:cond:02}, we can estimate 
\begin{align*}
    \mathbb{E} & \left[ \left\|\MC{G}_i(\Theta_i^{t,k}) - \nabla F_i(\Theta_i^{t,k}) - \left(\MC{G}_{i+1} (\Theta_{i+1}^{t,k})-\nabla F_{i+1}(\Theta_{i+1}^{t,k})\right)\right\|^2 \Big| \mathcal{F}^{t,k} \right] 
    \\
    &\leq 2\mathbb{E} \left[ \left\|\MC{G}_i(\Theta_i^{t,k}) - \nabla F_i(\Theta_i^{t,k})\right\|^2 \Big| \mathcal{F}^{t,k} \right] + 2\mathbb{E} \left[ \left\|\MC{G}_{i+1} (\Theta_{i+1}^{t,k})-\nabla F_{i+1}(\Theta_{i+1}^{t,k})\right\|^2 \Big| \mathcal{F}^{t,k} \right] 
    \\
    &\leq 4(\sigma_C^2 + \alpha^2N^2\sigma_S^2) = 4\Sigma,
\end{align*}
and we obtain from \eqref{Proof:Lemma:2:03}
\begin{align}\label{Proof:Lemma:2:04}
    \mathbb{E} \bigg[ \Big\|\Theta_i^{t,k+1} - \Theta_{i+1}^{t,k+1}\Big\|^2 \big| \mathcal{F}^{t,k} \bigg] \leq& \left\|\Theta_{i}^{t,k} - \Theta_{i+1}^{t,k}\right\|^2 + \eta^2\mathbb{E} \left[ \left\|\nabla F_i(\Theta_i^{t,k}) - \nabla F_{i+1}(\Theta_{i+1}^{t,k})\right\|^2 \Big| \mathcal{F}^{t,k} \right] \notag 
    \\
    &+4\eta^2\Sigma -2\eta \left\langle  \nabla F_i(\Theta_i^{t,k}) - \nabla F_{i+1}(\Theta_{i+1}^{t,k}) , \Theta_i^{t,k} - \Theta_{i+1}^{t,k} \right\rangle.
\end{align}

Moreover, thanks to the $L$-smooth Assumption \ref{ass:A2}, and using Young's inequality, we can be bound
\begin{align}\label{Proof:Lemma:2:05}    
    -\Big\langle &\nabla F_i(\Theta_i^{t,k}) - \nabla F_{i+1}(\Theta_{i+1}^{t,k}) , \Theta_i^{t,k} - \Theta_{i+1}^{t,k} \Big\rangle \notag
    \\
    =& -\left\langle \nabla F(\Theta_i^{t,k})-\nabla F(\Theta_{i+1}^{t,k}), \Theta_i^{t,k} - \Theta_{i+1}^{t,k} \right\rangle - \left\langle \nabla F_i(\Theta_i^{t,k}) -\nabla F(\Theta_i^{t,k}) , \Theta_i^{t,k} - \Theta_{i+1}^{t,k}  \right\rangle \notag
    \\
    &-\left\langle \nabla F_{i+1}(\Theta_i^{t,k}) - \nabla F(\Theta_{i+1}^{t,k}) , \Theta_{i}^{t,k} - \Theta_{i+1}^{t,k} \right\rangle \notag
    \\
    \leq & - \left\langle \nabla F(\Theta_i^{t,k}) - \nabla F(\Theta_{i+1}^{t,k}) , \Theta_i^{t,k} - \Theta_{i+1}^{t,k}  \right\rangle +2 \zeta \left\|\Theta_i^{t,k} - \Theta_{i+1}^{t,k}\right\| \notag
    \\
    \nonumber 
    \leq & -\dfrac{1}{L} \left\| \nabla F(\Theta_i^{t,k}) - \nabla F(\Theta_{i+1}^{t,k}) \right\|^2 + 2\zeta \left\| \Theta_i^{t,k} - \Theta_{i+1}^{t,k}  \right\|
    \\   
    \leq & -\dfrac{1}{L} \left\| \nabla F(\Theta_i^{t,k}) - \nabla F(\Theta_{i+1}^{t,k}) \right\|^2 + \dfrac{1}{2\eta \tau} \|\Theta_i^{t,k} - \Theta_{i+1}^{t,k}\|^2 + 2\eta \tau \zeta^2.
\end{align}
Substituting \eqref{Proof:Lemma:2:05} into \eqref{Proof:Lemma:2:04}, we then obtain
\begin{align}\label{Proof:Lemma:2:06}
    \mathbb{E} \bigg[ \Big\|\Theta_i^{t,k+1} - \Theta_{i+1}^{t,k+1}\Big\|^2 \big| \mathcal{F}^{t,k} \bigg] \leq& \left(1+\frac{1}{\tau}\right)\left\|\Theta_{i}^{t,k} - \Theta_{i+1}^{t,k}\right\|^2 + 4\eta^2 \tau \zeta^2 +4\eta^2\Sigma \notag
    \\
    &+ \eta^2\mathbb{E} \left[ \left\|\nabla F_i(\Theta_i^{t,k}) - \nabla F_{i+1}(\Theta_{i+1}^{t,k})\right\|^2 \Big| \mathcal{F}^{t,k} \right] \notag 
    \\
    & -\dfrac{2\eta}{L} \left\| \nabla F(\Theta_i^{t,k}) - \nabla F(\Theta_{i+1}^{t,k}) \right\|^2.
\end{align}
In the same manner, we can also bound
\begin{align*}
    \left\|\nabla F_i(\Theta_i^{t,k}) - \nabla F_{i+1}(\Theta_i^{t,k}) \right\|^2 \leq 3 \left\|\nabla F(\Theta_i^{t,k}) - \nabla F(\Theta_{i+1}^{t,k}) \right\|^2 + 6\zeta^2,
\end{align*}
thus obtaining from \eqref{Proof:Lemma:2:06}
\begin{align*}
    \mathbb{E} \bigg[ \Big\|\Theta_i^{t,k+1} - \Theta_{i+1}^{t,k+1}\Big\|^2 \big| \mathcal{F}^{t,k} \bigg] \leq& \left(1+\frac{1}{\tau}\right)\left\|\Theta_{i}^{t,k} - \Theta_{i+1}^{t,k}\right\|^2 + 4\eta^2 \tau \zeta^2 +4\eta^2\Sigma + 6\eta^2\zeta^2 
    \\
    & +\eta\left(3\eta-\frac{2}{L}\right) \left\| \nabla F(\Theta_i^{t,k}) - \nabla F(\Theta_{i+1}^{t,k}) \right\|^2
    \\
    \leq& \left(1+\frac{1}{\tau}\right)\left\|\Theta_{i}^{t,k} - \Theta_{i+1}^{t,k}\right\|^2 + 4\eta^2\Sigma + 10\tau\eta^2\zeta^2 
    \\
    & +\eta\left(3\eta-\frac{2}{L}\right) \left\| \nabla F(\Theta_i^{t,k}) - \nabla F(\Theta_{i+1}^{t,k}) \right\|^2,
\end{align*}
where we have used that $\tau\geq 1$. Finally, since $\eta$ satisfies the condition \eqref{eq:LR}, the last term of the above inequality is negative, and we get
\begin{align*}
    \mathbb{E} \bigg[ \Big\|\Theta_i^{t,k+1} - \Theta_{i+1}^{t,k+1}\Big\|^2 \big| \mathcal{F}^{t,k} \bigg] &\leq \left(1+\frac{1}{\tau}\right)\left\|\Theta_{i}^{t,k} - \Theta_{i+1}^{t,k}\right\|^2 + 4\eta^2\Sigma + 10\eta^2\zeta^2 
    \\
    &= \left(1+\frac{1}{\tau}\right)\mathbb{E}\left[\left\|\Theta_i^{t,k} - \Theta_{i+1}^{t,k}\right\|^2\bigg| \mathcal{F}^{t,k}\right] + 4\eta^2\Sigma + 10\tau\eta^2\zeta^2.
\end{align*}
Iterating, we get that 
\begin{align*}
    \mathbb{E} \bigg[ \Big\|\Theta_i^{t,k+1} - \Theta_{i+1}^{t,k+1}\Big\|^2 \big| \mathcal{F}^{t,0} \bigg] &\leq \frac{\left(1+\frac{1}{\tau}\right)^k-1}{\frac{1}{\tau}}\Big(4\eta^2\Sigma + 10\tau\eta^2\zeta^2\Big) 
    \\
    &= \left[\left(1+\frac{1}{\tau}\right)^k-1\right]\Big(4\tau\eta^2\Sigma + 10\tau^2\eta^2\zeta^2\Big) \leq 4\tau\eta^2\Sigma + 10\tau^2\eta^2\tau\zeta^2.
\end{align*}
From this last estimate, the inequality \eqref{client:drift:inequality} follows thanks to the the convexity of the norm. 
\end{proof}

\begin{proof}
The result follows by combining Theorem \ref{Theorem:Per:Round:Progress} with Lemma \ref{lemma:client_drift}. First of all, recall that from \eqref{estimate:Lemma:1} we have that 
\begin{align*}
    \mathbb{E}\left[\dfrac{1}{\tau} \sum_{k=1}^{\tau} F(\overline{\Theta}^{t,k}) - F(\Theta^T) \bigg|\mathcal{F}^{t,0} \right] \leq & \dfrac{1}{2\eta \tau} \left(\Big\|\overline{\Theta}^{t,0} - \Theta^T\Big\|^2 - \mathbb{E} \left[\Big\|\overline{\Theta}^{t,\tau} - \Theta^T\Big\|^2 \Big| \mathcal{F}^{t,0} \right] \right) 
    \\
    &+ \dfrac{L}{M\tau} \sum_{k=1}^{\tau} \sum_{i=1}^{M} \mathbb{E} \left[\Big\|\Theta_i ^{t,k} - \overline{\Theta}^{t,k}\Big\|^2 \Big| \mathcal{F}^{t,0}\right] +\dfrac{2\Sigma\eta}{M}.
\end{align*}
Moreover, recall that $\overline{\Theta}^{t,\tau}=\overline{\Theta}^{t+1,0}$ and 
\begin{align*}
    \Big\|\overline{\Theta}^{t,0} - \Theta^T\Big\|^2 = \mathbb{E} \left[\Big\|\overline{\Theta}^{t,0} - \Theta^T\Big\|^2 \Big| \mathcal{F}^{t,0} \right],
\end{align*}
so that the above estimate becomes  
\begin{align*}
    \mathbb{E}\left[\dfrac{1}{\tau} \sum_{k=1}^{\tau} F(\overline{\Theta}^{t,k}) - F(\Theta^T) \bigg|\mathcal{F}^{t,0} \right] \leq & \dfrac{1}{2\eta \tau} \mathbb{E} \left[\Big\|\overline{\Theta}^{t,0} - \Theta^T\Big\|^2 - \Big\|\overline{\Theta}^{t+1,0} - \Theta^T\Big\|^2 \Big| \mathcal{F}^{t,0} \right]
    \\
    &+ \dfrac{L}{M\tau} \sum_{k=1}^{\tau} \sum_{i=1}^{M} \mathbb{E} \left[\Big\|\Theta_i ^{t,k} - \overline{\Theta}^{t,k}\Big\|^2 \Big| \mathcal{F}^{t,0}\right] +\dfrac{2\Sigma\eta}{M}.
\end{align*}
Averaging over $t\in\inter{T-1}\cup\{0\}$, we therefore get
\begin{align*}
    \mathbb{E}\Bigg[\dfrac{1}{\tau T} & \sum_{t=0}^{T-1}\sum_{k=1}^{\tau} F(\overline{\Theta}^{t,k}) - F(\Theta^T) \bigg|\mathcal{F}^{t,0} \Bigg] 
    \\
    &\leq \dfrac{1}{2\eta \tau T} \Big\|\overline{\Theta}^{0,0} - \Theta^T\Big\|^2 +\dfrac{2\Sigma\eta}{M} + \dfrac{L}{M\tau T} \sum_{k=1}^{\tau} \sum_{i=1}^{M}\sum_{t=0}^{T-1} \mathbb{E} \left[\Big\|\Theta_i ^{t,k} - \overline{\Theta}^{t,k}\Big\|^2 \Big| \mathcal{F}^{t,0}\right]
    \\
    &= \dfrac{D^2}{2\eta \tau T} +\dfrac{2\Sigma\eta}{M} + \dfrac{L}{M\tau T} \sum_{k=1}^{\tau} \sum_{i=1}^{M}\sum_{t=0}^{T-1} \mathbb{E} \left[\Big\|\Theta_i ^{t,k} - \overline{\Theta}^{t,k}\Big\|^2 \Big| \mathcal{F}^{t,0}\right].
\end{align*}
Now, thanks to \eqref{client:drift:inequality}, we can estimate
\begin{align*}
    \dfrac{L}{M\tau T} \sum_{k=1}^{\tau} \sum_{i=1}^{M}\sum_{t=0}^{T-1} \mathbb{E} \left[\Big\|\Theta_i ^{t,k} - \overline{\Theta}^{t,k}\Big\|^2 \Big| \mathcal{F}^{t,0}\right] \leq 10L\tau^2\eta^2\zeta^2 + 4L\tau\eta^2\Sigma.
\end{align*}
Putting everything together, we therefore have that
\begin{align*}
    \mathbb{E}\left[\dfrac{1}{\tau T} \sum_{t=0}^{T-1}\sum_{k=1}^{\tau} F(\overline{\Theta}^{t,k}) - F(\Theta^T) \bigg|\mathcal{F}^{t,0} \right] \leq \dfrac{D^2}{2\eta \tau T} +\dfrac{2\Sigma\eta}{M} + 10L\tau^2\eta^2\zeta^2 + 4L\tau\eta^2\Sigma.
\end{align*}
\end{proof}
}

\end{document}